\documentclass[11pt, one side,emlines]{amsart}
\usepackage{amssymb,latexsym,xy,eucal,mathrsfs, graphicx, tikz}
\textwidth=15.2cm \textheight=24cm \theoremstyle{plain}
\newtheorem{lemma}{Lemma}[section]
\newtheorem{theorem}[lemma]{Theorem}
\newtheorem{proposition}[lemma]{Proposition}
\newtheorem{corollary}[lemma]{Corollary}

\theoremstyle{definition}
\newtheorem{example}[lemma]{Example}
\newtheorem{remark}[lemma]{Remark}

\newtheorem{definition}[lemma]{Definition}

\numberwithin{equation}{section}  \voffset
-55truept \hoffset -15truept
\begin{document}
\baselineskip 15truept
\title{Unitification of Weakly p. q.-Baer $*$-rings}
 \maketitle 
 \begin{center}
Anil Khairnar\\
 {\small Department of Mathematics, Abasaheb Garware College, Pune-411004, India.\\
  \email{\emph{anil.khairnar@mesagc.org; anil\_maths2004@yahoo.com}}}\\  
  \vspace{.4cm}
  B. N. Waphare\\  
 {\small Center for Advanced Studies in Mathematics,
Department of Mathematics,\\ Savitribai Phule Pune University, Pune-411007, India.}\\
 \email{\emph{bnwaph@math.unipune.ac.in; waphare@yahoo.com}} 
 \end{center} 
\begin{abstract} In this paper, we introduce a concept of weakly principally quasi-Baer $*$-rings
in terms of central cover. We prove that a $*$-ring is a principally quasi-Baer $*$-ring
if and only if it is weakly principally quasi-Baer $*$-ring with unity. 
A partial solution to the problem similar to unitification problem raised by S. K. Berberian is obtained.
\end{abstract}
\section{Introduction}
        Throughout this paper, $R$ denotes an associative ring.
                        A $*$-ring (ring with involution) $R$ is a ring equipped with an involution $x \rightarrow x^* $,
                        that is an additive anti-automorphism of period at most two.
                        An element $e$ in a $*$-ring $R$ is called a {\it projection} if it is self adjoint (i.e., $e=e^*$)
                        and {\it idempotent} (i.e., $e^2=e$). Let $S$ be a nonempty subset of $R$. We write
                        $r_R(S)= \{a \in R ~|~ s a = 0, ~\forall ~s \in S \}$, called the
                        {\it right annihilator} of $S$ in $R$, and $l_R(S)= \{a \in R ~|~  a s = 0, ~\forall ~s \in S \}$, 
                         called the
                        {\it left annihilator} of $S$ in $R$.
                      For projections $e, f$, we write $e \leq f$ in case
$e=ef$. By \cite [Proposition 1, page 4]{Ber}, the relation $\leq$
is a partial order relation on the set of projections in a
$*$-ring.

 In \cite{Kap}, Kaplansky introduced Baer rings and Baer $*$-rings to
abstract various properties of $AW^*$-algebras, von Neumann
algebras and complete $*$-regular rings.
 A $*$-ring $R$ is said to be a {\it Baer $*$-ring}, if for every nonempty
subset $S$ of $R$, $r_R(S)=eR$, where $e$ is a projection in $R$. 
     Herstien was convinced that the simplicity of the simple Lie algebra should follow solely 
     from the fact that $R=M_n(F)$ ($F$-a field) is a simple ring. 
     This led him in the $1950$'s to develop his Lie theory of arbitrary simple rings with involutions.
     Early motivation for studying $*$-rings came from rings of operators. 
     If $\mathscr{B}(H)$ is the set of all bounded linear operators on a (real or complex) Hilbert space $H$,
     then each $\phi \in \mathscr{B}(H)$ has an adjoint, $adj (\phi) \in \mathscr{B}(H)$,
     and $\phi \rightarrow adj(\phi)$ is an involution on the ring $\mathscr{B}(H)$.
                          
  According to Birkenmeier et al. \cite{Bir}, a $*$-ring $R$ is said to be a {\it quasi-Baer $*$-ring} 
  if the right annihilator of every
ideal of $R$ is generated, as a right ideal, by a projection in $R$.
 In the same paper \cite{Bir}, they provide
examples of Baer
rings which are quasi-Baer $*$-rings but not Baer $*$-rings.
                     In \cite{Ber}, Berberian developed the theory of Baer $*$-rings and Rickart (p.p.)
                     $*$-rings. 
                     
                     The following definitions are required in sequel, can be
                     found in \cite{Ber}.
                      A $*$-ring $R$ is said to be a {\it Rickart $*$-ring},
if for each $x \in R$, $r_R(\{x\})=eR$, where $e$ is a projection
in $R$.  For each element $a$ in a Rickart $*$-ring, there is unique projection
$e$ such that $ae=a$ and $ax=0$ if and only if $ex=0$, called the {\it
right projection} of $a$ denoted by $RP(a)$. Similarly, the left
projection $LP(a)$ is defined for each element $a$ in a Rickart $*$-ring $A$.
 A $*$-ring $R$ is said to be a {\it weakly Rickart  $*$-ring}, if for any $x \in R$, there exists a projection $e$
  such that (1) $xe=x$, and (2) if $xy=0$ then $ey=0$.  
 Let $R$ be a $*$-ring and $x \in R$, we say that $x$
possesses a {\it central cover} if there exists a smallest central
projection $h$ such that $hx=x$. If such projection $h$ exists,
then it is unique, and is called the central cover of $x$, denoted
by $h=C(x)$ or $h=C_R(x)$.

 In \cite{Bir2}, Birkenmeier  introduced principally quasi-Baer (p.q.-Baer) rings as a generalization of quasi-Baer rings.
  Birkenmeier et al. \cite{Bir4} introduced principally quasi-Baer
 (p.q.-Baer) $*$-rings. A $*$-ring $R$ is said to be {\it a p.q.-Baer $*$-ring} if,
   for every principal right ideal $aR$ of $R$, $r_R(aR)=eR$, where $e$ is a projection in
$R$. 
From the above definition, it follows that $l_R(aR)=Rf$ for a suitable projection $f$.
 Note that, $R$ is a quasi-Baer $*$-ring if and only if $M_n (R)$ ($n \times n$ matrix ring with $*$ transpose involution)
 is a quasi-Baer (hence p.q.-Baer) $*$-ring for all $n\geq 1$  \cite[Proposition 2.6]{Bir1}. 
 In general this statement is not true for Baer $*$-rings  \cite[Theorem 6]{Tha}. 
 Thus the class of  p.q.-Baer $*$-rings is very much larger than that of Baer $*$-rings.
 
    In \cite{Ber}, it is proved that every element of a Baer
    $*$-ring has a central cover.  In the second section of this paper, 
    we extend this result for p.q.-Baer $*$-rings. In fact,    
    we generalize results of Baer $*$-rings to  p.q. Baer
    $*$-rings. We have given an example of a $*$-ring
$R$, which is a Rickart $*$-ring as well as a p.q.-Baer $*$-ring,
and an element $x \in R$ such that $RP(x)$ is not equal to $C(y)$
for any $y \in R$.  

  In the third section, we introduce
 the concept of weakly p.q.-Baer $*$-rings. We provide an example of an abelian
p.q.-Baer $*$-ring which is not a weakly Rickart $*$-ring.
   Efforts are taken to obtain the exact difference between
  p.q.-Baer $*$-rings and weakly p.q.-Baer $*$-rings.
  
  Berberian \cite[page 33]{Ber}, raised the following Problem.\\
   {\bf Problem 1:} Can every weakly Rickart $*$-ring be
  embedded in a Rickart $*$-ring? with preservation of $RP$'s?\\
 \indent Berberian has given a partial solution to this problem.
 In \cite{Wap}, Thakare and Waphare gave more general partial solution to Problem 1.
 Till date this problem is open. In view of Problem 1, it is
natural to raise the following problem.\\
{\bf Problem 2:} Can every weakly p.q.-Baer $*$-ring be
  embedded in a p.q.-Baer $*$-ring? with preservation of central covers?\\
 \indent In the last section we give the partial solution of Problem 2, analogous
 to the partial solution of Problem 1 given by Thakare and Waphare in
 \cite{Wap}.
  \section{Central Cover}      
In this section, we extend the results of Baer $*$-rings (proved in \cite{Ber}) to p.q.-Baer
$*$-rings.  Also, we study the properties of a central cover of an element in a p.q.-Baer $*$-ring.\\
    \indent According to Birkenmeier et al. \cite{Bir3}, the involution of a $*$-ring
$R$ is  {\it semi-proper} if $aRa^*\neq 0$ for every nonzero
element $a\in R$.
\begin{proposition} \label{pr1} If $R$ is a p.q.-Baer $*$-ring,
then $R$ has the unity element and the involution of $R$ is
semi-proper.
\end{proposition}
\begin{proof} Since $R$ is a p.q.-Baer $*$-ring, there exists a
projection $e\in R$ such that $r_R(0R)=R=eR$. As $e \in eR$, we
have $e=er$ for some $r \in R$. Let $x \in R=eR$, so $x=er'$ for
some $r' \in R$. Clearly $ex=e(er')=e^2r'=er'=x$. Hence $e$ is
the left unity in $R$. Since $R$ is a $*$-ring, $e$ is also the
right unity in $R$. Thus $e$ is an unity in $R$. By \cite[Proposition 2]{Usa}, $*$ is semi-proper.
\end{proof}
\begin{remark}\label{rem01} Note that in a p.q.-Baer $*$-ring, the
projection which generates the right annihilator of a principal
right ideal is central.
 For this, let $k$ be a projection in a p.q.-Baer
$*$-ring $R$ such that $r_R(zR)=kR$ for some $z \in R$. Let $r \in
R$. Since $rk \in r_R(zR)=kR$, we have $rk=kr'$ for some $r' \in
R$. Hence $krk=kr'=rk$. Thus $kr=(r^*k)^*=(kr^*k)^*=krk=rk$.
\end{remark}
\begin{theorem} \label{th002'} Let $R$ be a p.q.-Baer $*$-ring and $x\in R$.
Then $x$ has a central cover $e \in R$. Further,
  $xRy=0$ if and only if $yRx=0$ if and only if $ey=0$.\\
That is $r_R(xR)=r_R(eR)=l_R(Rx)=l_R(Re)=(1-e)R=R(1-e)$.
\end{theorem}
\begin{proof} Let $R$ be a p.q.-Baer $*$-ring and $x\in R$.
First we prove that $x$ has a central cover. Since $R$ is a
p.q.-Baer $*$-ring, there exists a projection $g \in R$ such that
$r_R(xR)=gR$. Let $e=1-g$. Consider $xe=x(1-g)=x-xg=x$. By Remark
\ref{rem01}, $e$ is central. To prove $e$ is the smallest central
projection with the property $xe=x$, suppose $e'\in R$ be a
central projection such that $xe'=x$. So $x(1-e')=0$. Thus
$xR(1-e')=0$. Consequently $(1-e') \in r_R(xR)=(1-e)R$. It follows
that $1-e'= (1-e)t$ for some $t \in R$. This gives
$(1-e')(1-e)=(1-e)t=(1-e')$, that is $(1-e') \leq (1-e)$.
Therefore
 $e \leq e'$. Thus $e=C(x)$.\\
 \indent If $xRy=0$, then $y \in r_R(xR)=gR $. Hence $y=gr$ for some $r \in
R$. As $g$ is a projection, we have $gy=g^2r=gr=y$. Therefore
$ey=(1-g)y=y-gy=0$. Conversely, suppose $ey=0$. Then
$y=y-ey=(1-e)y=gy$. Since $g \in r_R(xR) $, we have $xry=xrgy=0$ for any
$r \in R$. This yields $xRy=0$.
 Similarly, there
exists a central projection $f \in R$ such that, $fx=x$; and
$yRx=0$ if and only if $yf=0$.\\
 \indent Now we prove that $e=f$. Since $xR(e-f)=0$, we have $e(e-f)=0$, that is
$e=ef$. Similarly, since $(e-f)Rx=0$, we have $(e-f)f=0$, that is,
$ef=f$. This yields $e=f$.
Therefore $xRy=0$ if and only if $yRx=0$ if and only if $ey=0$.\\
 \indent Observe that, $r_R(xR)=(1-e)R$ and
$l_R(Rx)=R(1-e)$. As $e$ is central, $1-e$ is also central.
Consequently, $r_R(xR)=l_R(Rx)=R(1-e)$. To prove $r_R(eR)=R(1-e)$,
let $z \in r_R(eR)$. Then $eRz=0$, so
 $ez=0=ze$. Hence $z=z(1-e)\in R(1-e)$. Thus
$r_R(eR)\subseteq R(1-e)$. Since $1-e$ is central, we have $R(1-e)
\subseteq r_R(eR) $. This gives $r_R(eR)=R(1-e)$. In nutshell, we
get $r_R(xR)=r_R(eR)=l_R(Rx)=l_R(Re)=(1-e)R=R(1-e)$.
\end{proof}
 The
following example shows that $RP$'s in Rickart $*$-rings are not
necessarily central covers in p.q.-Baer $*$-rings.
\begin{example}\label{ex01} Let $A=M_2$($\mathbb{Z}_3$), which is a Baer $*$-ring (hence a p.q.-Baer $*$-ring and a Rickart $*$-ring) with transpose
as an involution. By \cite[Exercise 17, page 10]{Ber}, the set of
all projections $P(A)$ in $A$ is, $P(A)=\{0, 1, e, f, 1-e, 1-f
\}$, where $e= \left(
\begin{array}{cc}
  1 & 0 \\
  0 & 0 \\
\end{array}
\right)$,
 $f= \left(
\begin{array}{cc}
  2 & 2 \\
  2 & 2 \\
\end{array}
\right)$.
 Out of these only $0$ and $1$ are central projections in
$A$. Note that the set of all right projections in $A$ is $P(A)$,
whereas the set of all central covers  in $A$ (being central) is
$\{0, 1\}$. Therefore there is an element $x \in A$ such that
$RP(x)$ is not equal to $C(y)$ for any $y \in A$. In particular,
$r_A(x) \neq r_A(xA)$ for some $x \in A$.
\end{example}
Now we provide properties of a central cover of an element in
p.q.-Baer $*$-rings.
\begin{theorem} \label{th1} If $R$ is a p.q.-Baer $*$-ring, then\\
(1) $C(x)=C(x^*)$; (2) $xRy=0$ if and only if $C(x)C(y)=0$.
\end{theorem}
\begin{proof} (1): Let $e=C(x^*)$ and $f=C(x)$.
 By the definition of central cover, $x^*e=x^*$ and hence $ex=x=fx$. 
 Therefore $ex-fx=0$. This gives $rex-rfx=0$ for all $r \in R$.
 Since $e$ and $f$ are central, $erx-frx=0$. 
 That is $(e-f)rx=0$ for all $r \in R$. It follows that $(e-f)Rx=0$. By Theorem \ref{th002'}, we have
 $(e-f)f=0$. Hence $ef=f$. Similarly, if
 $fx=x$ then $x^*f=x^*=x^*e$, this yields $ef=e$. Thus $e=f$.
\vspace{3mm} \\
  (2): Let $e=C(x)$ and $f=C(y)$. If $xRy=0$, then $ey=0$.
 As $e$ is central, $eRy=0$. Hence $ef=0$. Conversely, suppose
 $ef=0$. We have $eRy=0$, that is, $ey=0$. Thus $xRy=0$.
 \end{proof}
   Projections $e,~f$ in a p.q.-Baer $*$-ring are
called {\it very orthogonal} if $C(e)C(f)=0$ (equivalently, in view of
Theorem \ref{th1}, $eRf=0$).
\begin{corollary} \label{cr3} Let $R$ be a p.q.-Baer $*$-ring whose
only central projections are $0$ and $1$. If $x, y \in R$, then
$xRy=0$ if and only if $x=0$ or $y=0$.
\end{corollary}
\begin{proof} If $x \neq 0$ and $y \neq 0$, then $C(x)=C(y)=1$.
Therefore $C(x) C(y) \neq 0$. By Theorem \ref{th1}, $xRy \neq 0$.
If $x=0$ or $y=0$ then by the definition of a central cover, we have
$C(x)=0$ or $C(y)=0$. That is $C(x)C(y)=0$. So by Theorem
\ref{th1}, $xRy=0$.
\end{proof}
According to  \cite{Usa}, a $*$-ring $R$ is said to satisfy the {\it $*$-Insertion of Factors Property}
(simply, $*$-$IFP$) if $ab = 0$ implies $aRb^* = 0$ for all $a, b \in R$.
\begin{corollary} Let $R$ be a p.q.-Baer $*$-ring.
If $C(xy)=C(x)C(y)$ for all $x,y \in R$, then $R$ is a Rickart $*$-ring.
Moreover $RP(x)=C(x)$ for all $x \in R$. 
\end{corollary}
\begin{proof} Let $x,y \in R$ be such that $xy=0$. This gives $C(xy)=0$.
        By assumption, we have $C(x) C(y)=0$. 
       Clearly $C(y)=C(y^*)$, therefore $C(x) C(y^*)=0$. By Theorem \ref{th1}, $xRy^*=0$.
        Therefore $R$ satisfies $*$-$IFP$.  By \cite[Proposition 9]{Usa}, $R$ is reduced and hence a Rickart $*$-ring.
        Also, observe that $r_R(x)=r_R(xR)$ and hence $RP(x)=C(x)$ for all $x \in R$.  
\end{proof}  
\begin{proposition} \label{pr4} Let $R$ be a p.q.-Baer $*$-ring
and suppose $(e_i)$ is a family of projections that has a supremum, say
$e$. If $x\in R$, then $xRe=0$ if and only if $xRe_i=0$ for all
$i$.
\end{proposition}
\begin{proof} Let $x \in R$ and $C(x)=e'$. Suppose $xRe=0$.
As $C(x)=e'$, $e'e=0$, which yields $e(1-e')=e$. Therefore $e \leq
(1-e')$. Since $e_i\leq e$ for all $i$, we have $e_i \leq (1-e')$
for all $i$. Hence $e_i(1-e')=e_i$. This gives
$e_ie'=0$, that is, $e'e_i=0$ for all $i$. As $C(x)=e'$, $xRe_i=0$
for all $i$. By reversing the steps, we get the converse part.
\end{proof}
Let $R$ be a ring and
$S\subseteq R$ be nonempty, then $S'=\{x \in R ~|~xs=sx,~ \forall ~
s\in S \}$. A subring $B$ of a $*$-ring $R$ is said to be $*$-subring, if $x\in B$ imply $x^*\in B$.
\begin{theorem} \label{th7} Let $R$ be a p.q.-Baer $*$-ring and
$B$ be a $*$-subring of $R$ such that $B=(B')'$. If $(e_i)$ is a
family of central projections in $B$ that possesses a supremum, say $e
\in R$, then $e \in B$.
\end{theorem}
\begin{proof} Since $e_i \in B=(B')' $, we have $e_iy=ye_i$ for
all $y \in B'$. Clearly
$e_i(y-ye)=e_iy-e_iye=e_iy-ye_ie=e_iy-ye_i$ (as $e_i \leq e$).
Since $e_iy=ye_i$, we get $e_i(y-ye)=0$. Moreover, each $e_i$ is central,
$e_iR(y-ye)=0$. By Proposition \ref{pr4}, $e R(y-ye)=0$. Hence
$e(y-ye)=0$, that is $ey=eye$ for all $y \in B'$. If $y \in B'$,
then $y^* \in B'$. Therefore $ey^*=ey^*e$, which gives $ye=eye$.
Consequently, $ey=ye$ for all $y \in B'$. Thus $e \in (B')'=B$.
\end{proof}

\section{Weakly $p.q.$-Baer $*$-rings}

In this section, we introduce weakly p.q.-Baer $*$-rings.
 By Proposition \ref{pr1}, p.q.-Baer $*$-rings has unity. What
 about the unit-less case? The answer, in principle, is to modify
 the arguments or attempt to adjoin the unity element. First, let us
 review the unit case.
\begin{proposition} \label{pr301} A $*$-ring $R$ is a p.q.-Baer
$*$-ring if and only if \\
(a) $R$ has the unity element.\\
(b) For each $x \in R$, there exists a central projection $e\in R$
such that $r_R(xR)=r_R(eR)$.
\end{proposition}
\begin{proof}  If $R$ is a p.q.-Baer
$*$-ring, then by Proposition \ref{pr1}, $R$ has the unity element
and by Theorem \ref{th002'}, the condition $(b)$ is satisfied.
Conversely suppose that the conditions $(a)$ and $(b)$ are satisfied. Since $e$
is a central projection, $1-e$ is a central projection. So $(1-e)R
\subseteq r_R(eR)$. Let $y \in r_R(eR)$. Then $ey=0$. Hence
$y=y-ey=(1-e)y$, so $y \in (1-e)R$. Therefore $r_R(eR) \subseteq
(1-e)R$. This yields $r_R(xR)=r_R(eR) = (1-e)R$. Thus $R$ is a
p.q.-Baer $*$-ring.
\end{proof}
 The following example shows that the condition $(a)$ in Proposition \ref{pr301} is necessary.
\begin{example} [Exercise 1, page 32, \cite{Ber}] Let $R$ be a $*$-ring with
$R^2=\{0\}$. Then $0$ is the only projection in $R$ and
$r_R(xR)=r_R(0R)=R$ for every element $x \in R$, but $x 0 \neq x$,
when $x \neq 0$.
\end{example}
 To study the unit-less case, we introduce the concept of weakly p.q.-Baer $*$-rings as follows.
\begin{definition} A $*$-ring $R$ is said to be a weakly
p.q.-Baer $*$-ring, if every $x\in R$ has a central cover $e \in
R$ such that,
  $xRy=0$  if and only if $ey=0$.
  \end{definition}
  The following is an example of an abelian
p.q.-Baer $*$-ring which is not a Rickart $*$-ring, see \cite{Kim}.
\begin{example}\label{ex001} Let $ R={\displaystyle\left\{\begin{bmatrix}  a  & b \\  c &   d   \end{bmatrix} \in M_2(\mathbb Z)~|~
                 a \equiv d,~b \equiv 0,~ {\rm and} ~c \equiv 0 ~(mod~2)\right \}}$.
                 Consider  involution $*$ on $R$ as the transpose of the matrix.
                 In \cite[Example 2(1)]{Kim}, it is shown that $R$ is neither right p.p. nor left
                 p.p. (hence not a Rickart $*$-ring) but
                 $r_R(uR) = \{0\}=0R$ for any nonzero element $u \in
R$. Therefore $R$ is a p.q.-Baer $*$-ring.
\end{example}
{\bf Note that the ring $R$ in above example is weakly p.q.-Baer $*$-ring but not weakly Rickart $*$-ring 
(as $R$ has unity).}

  By \cite[Exercise 6, page 32]{Ber}, a weakly Rickart
  $*$-ring with finitely many elements is a Baer $*$-ring.
  We give an example of a p.q.-Baer $*$-ring (hence a weakly p.q.-Baer
  $*$-ring) with finitely many elements which is not a Baer
  $*$-ring. First, we recall the following corollary.
\begin{corollary} [\cite{Tha}, Corollary 7] \label{c031'} (i) $M_n(\mathbb Z_m)$ is a Baer
$*$-ring for $n \geq 2$ if and only if $n=2$ and $m$ is a square
free integer whose every prime factor is of the form $4k+3$.\\
(ii) $\mathbb Z_m$ is a Baer $*$-ring if and only if $m$ is a
square free integer.
\end{corollary}
\begin{example} \label{ex301} Let $R=M_2$($\mathbb Z_6$), which is a $*$-ring with transpose
as an involution. By Corollary \ref{c031'}(ii), $\mathbb{Z}_6$ is
a Baer $*$-ring and hence $\mathbb{Z}_6$ is a quasi-Baer $*$-ring.
By \cite[Proposition 2.6]{Bir}, $R$ is a quasi-Baer $*$-ring and
hence $R$ is a p.q.-Baer $*$-ring. Note that $R$ contains finitely
many elements. By Corollary \ref{c031'}(i), $R$ is not a Baer
$*$-ring.
\end{example}
The following theorem leads to a characterization of weakly
p.q.-Baer $*$-rings.
\begin{theorem} \label{nt301} Let $R$ be a $*$-ring (not necessarily with the unity), $x\in R$ and $e \in R$ be a
central projection. Then the following statements are equivalent.\\
(a) $C(x)=e$; and $xRy=0$  if and only if $ey=0$.\\
(b) $xe=x$; and $xRy=0$  if and only if $ey=0$.
\end{theorem}
\begin{proof} Observe that, if the statement $(a)$ holds then the statement $(b)$ holds.
             Conversely suppose statement $(b)$ holds. To prove statement $(a)$, it is sufficient to
               prove $e$ is the smallest central projection with $xe=x$.
               Let $e' \in R$ be the central projection such that
               $xe'=x$. Therefore $x(e-e')=0$. As $e-e'$ is
               central, we have $xR(e-e')=0$. By the assumption, we
               get $e(e-e')=0$. Thus $e \leq e'$.
\end{proof}
    \begin{corollary} \label{nc301} $R$ is a weakly p.q.-Baer $*$-ring if and only
  if for every $x\in R$ there exists a central projection $e \in R$ such
  that
  (1) $xe=x$; (2) $xRy=0$ if and only if $ey=0$.
  \end{corollary}
 We give one more characterization of a weakly p.q.-Baer
$*$-rings as follows.
\begin{theorem} \label{th303'} The following conditions on a
$*$-ring $R$ are equivalent:\\
(a) $R$ is a weakly p.q.-Baer $*$-ring.\\
(b) $R$ has a semi-proper involution, and for each $x \in R$,
there exists a central projection $e$ such that $r_R(xR)=r_R(eR)$.
\end{theorem}
\begin{proof} $(a)\Rightarrow (b)$: Let $a \in R$ be such that
$aRa^*=0$. There exists a projection $e \in R$ such that (1)
$e=C(a)$; (2) $aRy=0$ if and only if $ey=0$. As $aRa^*=0$, we have
$ea^*=0$, that is
 $ae=0$. This gives $a=0$. Hence $*$ is semi-proper. Let $y \in
r_R(xR)$. Consequently, $xRy=0$. So $ey=0$. It follows that $eRy=0$.
This yields $y \in r_R(eR)$. Therefore $r_R(xR) \subseteq
r_R(eR)$. Let $z \in r_R(eR)$. Then $eRz=0$. Consider
$xRz=xeRz=0$. This gives $z \in r_R(xR)$. Hence $r_R(eR) \subseteq
r_R(xR) $. Thus $r_R(xR)=r_R(eR)$.
\vspace{3mm} \\
$(b)\Rightarrow (a)$: Let $x\in R$ and $e$ be the central
projection such that $r_R(xR)=r_R(eR)$. Since for all $y \in R$,
$eR(y-ey)=0$, we have $xR(y-ey)=0$. Put $y=x^*$, we get
$xR(x^*-ex^*)=0$. Let $(x-xe)r(x^*-ex^*) \in (x-xe)R(x^*-ex^*)=
(x-xe)R(x-xe)^* $, where $r \in R$. Clearly
$(x-xe)r(x^*-ex^*)=(x-xe)rx^* - (x-xe)rex^*=xr (x^*-ex^*)=0$.
Therefore $(x-xe)R(x-xe)^*=0$. As $*$ is semi-proper, we have
$x-xe=0$. Hence $x=xe$. If $xRy=0$ then $y \in r_R(xR)=r_R(eR)$,
so $eRy=0$. Consequently $eey=0$, that is, $ey=0$. Thus, by
Corollary \ref{nc301}, $R$ is a weakly p.q.-Baer $*$-ring.
\end{proof}
\begin{theorem} \label{th302'} The following conditions on a
$*$-ring $R$ are equivalent:\\
(a) $R$ is a p.q.-Baer $*$-ring.\\
(b) $R$ is a weakly p.q.-Baer $*$-ring with the unity.
\end{theorem}
\begin{proof} $(a)\Rightarrow (b)$: Follows from Proposition
\ref{pr1} and Theorem \ref{th002'}.
\vspace{3mm} \\
$(b)\Rightarrow (a)$: By the definition of a weakly p.q.-Baer
$*$-ring,
 for each $x \in R$, there exists a projection $e$ such
that (1) $C(x)=e$; (2) $xRy=0$ if and only if $ey=0$. By similar
steps as in Theorem \ref{th303'}, we have $r_R(xR)=r_R(eR)$. By
Proposition \ref{pr301}, $R$ is a p.q.-Baer $*$-ring.
\end{proof}

   \section {Unitification}

   Theorem \ref{th302'} distinguishes 
   p.q.-Baer $*$-rings and weakly p.q.-Baer $*$-rings. With
   understanding of this difference, we try to adjoin the unity to
   weakly p.q.-Baer $*$-rings (that is an unitification of weakly p.q.-Baer $*$-rings). This gives a partial solution of
   {\bf Problem 2}. \\
  \indent First, we define an unitification of a $*$-ring as
   follows.
 \begin{definition} \label{def1} Let $R$ be a $*$-ring. We say that $R_1$ is an unitification of $R$, if there exists an auxiliary ring $K$, called the ring of
   scalars (denoted by $ \lambda, \mu,...$) such that,\\
   1) $K$ is an integral domain with involution (necessarily
  semi-proper), that is, $K$ is a commutative $*$-ring with unity
   and without divisors of zero (the identity involution is
   permitted),\\
   2) $R$ is a $*$-algebra over $K$ (that is, $R$ is a left $K$-module such that,
   identically
   $1a=a,~\lambda (ab)=(\lambda a)b=a (\lambda b),~ and~(\lambda a)^*=\lambda ^*
   a^*$).\\
     Define $R_1=R \oplus K$ (the additive group direct sum), thus
   $(a, \lambda)=(b, \mu)$ means, by the definition that $a=b$ and $\lambda
   =\mu$, and addition in $R_1$, is defined by the formula $(a, \lambda)+(b, \mu)=(a+b, \lambda +
   \mu)$. Define  $(a, \lambda)(b, \mu)=(ab+ \mu a+ \lambda b, \lambda
   \mu)$, $\mu (a, \lambda)=(\mu a, \mu \lambda)$, $(a, \lambda)^*=(a^*, \lambda
   ^*)$. Evidently $R_1$ is also a $*$-algebra over $K$, has unity
   element $(0, 1)$ and $R$ is a
   $*$-ideal in $R_1$.
   \end{definition}
    The following lemmas are elementary facts about unitification $R_1$ of a $*$-ring $R$.
\begin{lemma} \label{lm1} With notations as in Definition
\ref{def1}, if an involution on $R$ is semi-proper, then so is the
involution of $R_1$
\end{lemma}
\begin{proof} Let $(a, \lambda)\in R_1$ be such that $(a, \lambda) R_1 (a,
\lambda)^*=(0,0) $. So $(a, \lambda) (0,1) (a^*,
\lambda^*)=(aa^*+\lambda^* a+ \lambda a^*, \lambda
\lambda^*)=(0,0) $, This gives $ \lambda \lambda^*=0$. As
an involution of $K$ is semi-proper, we have $\lambda=0$. Therefore,
for any $r\in R$, $(a, 0) (r, 0) (a^*, 0)=(0, 0)$. That is
$(ara^*, 0)=(0, 0)$. So $ara^*=0 $ for all $r \in R$. Hence
$aRa^*=0$. Also an involution on $R$ is semi-proper, we have $a=0$. So
$(a, \lambda)=(0,0)$. Thus involution of $R_1$ is semi-proper.
\end{proof}
\begin{lemma} \label{lm3} With notations as in Definition
\ref{def1}, let $x \in R$ and let $e$ be a projection in $R$. Then $C(x)=e$
 in $R$ if and only if $C((x,0))=(e,0)$ in
$R_1$.
\end{lemma}
\begin{proof} Let $x\in R$. Suppose $e=C(x)$. So (1) $xe=x$;
  (2) $xRy=0$ if and only if $ey=0$. Consider $(x,0)(e,0)=(xe, 0)=(x,
0)$. Suppose $(x,0)R_1(b, \lambda)=(0,0)$. Hence
$(0,0)=(x,0)R_1(b, \lambda)= (xe,0)R_1(b, \lambda)=
(x,0)(e,0)R_1(b, \lambda)=(x,0)R_1(e,0)(b, \lambda)=(x,0)R_1(eb+
\lambda e, 0)$. This yields $xR(eb+ \lambda e)=0$. Thus $e(eb+
\lambda e)=0$. Consequently $eb+ \lambda e =0 $, that is, $(e,0)(b,
\lambda)=(0,0)$.
 Therefore, by Theorem \ref{nt301}, $(e,0)$ is a central cover of $(x,0)$ in $R_1$.
\end{proof}
\begin{proposition} \label{pr304} In a weakly p.q.-Baer
$*$-ring $R$, $xRy=0$ if and only if $C(x)C(y)=0$.
\end{proposition}
\begin{proof} Follows in the same way as the proof of the Part (2) of the
Theorem \ref{th1}.
\end{proof}
\begin{theorem}  \label{th301} Let $R$ be a weakly p.q.-Baer
$*$-ring. If $e$ and $f$ are projections in $R$, then there exists
a central projection $g$ such that $e \leq g$ and $f \leq g$ (i.e.
$g$ is an upper-bound of the set $\{e, f \}$).
\end{theorem}
\begin{proof} Suppose that $f$ is a central projection in $R$. Let
$g=f+C(e-ef)$ and $h=C(e-ef)$, that is, $g=f+h $. As $f$ is a central
projection, $C(f)=f$ and $fR(e-ef)=R(e-ef)f=\{0\}$. By Proposition
\ref{pr304}, $C(f) C(e-ef)=0 $, consequently $f h=0$. So $g=f+h$
is a central projection in $R$. Consider $fg=f(f+h)=f+fh=f$. Hence
$f \leq g$. Since $h=C(e-ef)$, $(e-ef)h=e-ef$. Therefore
$0=(e-ef)h-(e-ef)=eh-e(fh)-e+ef=eh-e+ef$. Thus $e=ef+eh=e(f+h)=eg$.
This yields $e \leq g$. Consequently $g$ a central projection
which is an upper-bound of $\{ e,f \}$. Moreover $g$ is the least
upper bound as a central projection in the set of all upper bounds
which are central projections. Now suppose $e$ and $f$ are any
two projections (not necessarily any of them to be central) in
$R$. Let $g'$ be a central projection in $R$ (we can take $g'=0$).
Let $e'$ be a central projection which is an upper bound of $\{e,
g' \}$ and $f'$  be a central projection which is an upper bound
of $\{f, g' \}$. Let $g''$  be a central projection which is an
upper bound of $\{e', f' \}$. Therefore $g''$ is an upper-bound of
$\{e,f \}$.
\end{proof}
    Now, we provide a partial solution to the unitification problem of weakly p.q.-Baer $*$ rings.
\begin{theorem} \label{th304} A weakly p.q.-Baer $*$-ring $R$ can be embedded in a
p.q.-Baer $*$-ring, provided there exists, a ring $K$ such that,\\
(i) $K$ is an integral domain with involution,\\
(ii) $R$ is a $*$-algebra over $K$,\\
(iii) For any $\lambda \in K-\{0\}$ there exists a projection
$e_{\lambda} \in R$ that is an upper bound for the central covers
of the right annihilators of $\lambda$, that is, for $t \in R$, if
$\lambda~t=0$ then $C(t) \leq e_{\lambda}$.
\end{theorem}
\begin{proof} Let $R_1=R\oplus K$ (the additive group direct sum),
with operations in Definition \ref{def1}. First we prove that for
any $a \in R$ and a nonzero scalar $\gamma \in K$, there exists a
greatest central projection $g$ in $R$ such that $ag=\gamma g$. By
the condition (iii), there exists a projection $e_{\gamma} \in R$ that
is an upper bound for the central covers of the right annihilators
of $\gamma$. Let $e_0=C(a)$.  By Theorem \ref{th301}, there exists
a central projection $e$ which is an upper bound of $\{e_0,
e_{\gamma} \}$. Therefore $e$ is a projection in $R$ with $e_0
\leq e$, that is, $e_0=e_0e$. As $e_0=C(a)$, we have $ae_0=a$. It
follows that $ae_0e=ae$. Hence $a=ae$. Since $e$
is central, $a=eae \in eRe$. Therefore $a-\gamma e=eae - e \gamma
e= e(a- \gamma e) e \in eRe$. Let $h= C(a- \gamma e)$. So $(a-
\gamma e)h= a- \gamma e$. As $(a- \gamma e)e=a- \gamma e$, we have
$h \leq e$, i.e., $h=he$. Observe that $g=e-h$ is a central
projection. Consider $ag - \gamma e g= a (e-h) - \gamma e (e-h)
=ae - ah - \gamma e + \gamma eh= (a- \gamma e)- (a - \gamma
e)h=(a- \gamma e) - (a- \gamma e)=0 $. Also $\gamma eg= \gamma e
(e-h)=\gamma e - \gamma eh=\gamma e - \gamma he = \gamma e -
\gamma h= \gamma (e-h)= \gamma g$. This yields $0= ag - \gamma eg
= ag- \gamma g $. Thus $ag= \gamma g$.\\
\indent Now we prove maximality of
$g$. Let $k$ be a central projection in $R$ such that $ak= \gamma
k$, that is, $ak - \gamma k=0$. Since $a=ea$, $eak- \gamma k =0$.
Hence $e \gamma k - \gamma k =0$. Consequently $\gamma (ek-k)=0$.
Let $f= C(ek-k)$. So $f$ is a central projection in $R$. Therefore
$(\gamma f) (ek-k)= f \gamma (ek-k)=0$. It follows that
$(ek-k)(\gamma f)=0$, which further yields $(ek-k)R(\gamma f)=0$. Hence
 $f(\gamma f)=0$, and consequently $ \gamma f =0$. This gives
$C(f) \leq e_{\gamma} \leq e$. So $f \leq e$, that is,
 $f=fe$. As $e$ is a central projection and $f=
C(ek-k)$, $(ek-k)Re=0$. Therefore $fe=0$, thus $f=0$. Hence
$ek-k=f(ek-k)=0$. So $ek=k$, hence $k \leq e$. Since $ak =\gamma
k$ and $k$ is central, we have $(a-\gamma e) R k=0$. Further $h=C(a-
\gamma e)$, gives $h k=0$. Consider $kg=k (e-h)=ke-kh=ke-0=k$, that is,
 $ k \leq g$. Thus $g$ is the greatest central projection
such that $a g= \gamma g$.\\
  \indent To show $R_1$ is a p.q.-Baer
$*$-ring, it is sufficient to prove that for every element  $x \in
R_1$ there exists a central projection $e \in R_1$ such that (1)
$xe=x$; and (2) $xRy=0$ if and only if $ey=0$ (because of
Corollary \ref{nc301} and Theorem \ref{th302'}). Let $x=(a,
\lambda) \in R_1$. If $\lambda =0$, then by Lemma \ref{lm3}, $x$
has a central cover. Suppose $\lambda \neq 0$. By the above
discussion, there exists a greatest central projection $g$ in $R$
such that $ag=(- \lambda)g$. So, $xg=(a, \lambda) (g,
0)=(ag+\lambda g, 0)=(0,0)$. Let $e=1-g=(0,1)-(g,0)=(-g,1)$. We
prove that $C(x)=e$ in $R_1$. Consider $xe=(a, \lambda)(-g,
1)=(-ag+a-\lambda g, \lambda)=(\lambda g+a - \lambda g,
\lambda)=(a, \lambda)=x$. Suppose $(a, \lambda) R_1 (b,
\mu)=(0,0)$. So $(a, \lambda) (0,1) (b, \mu)=(0,0)$. Hence $(a,
\lambda)(b, \mu)=(0,0)$, that is, $(ab+\mu a+ \lambda b, \lambda
\mu)=(0,0)$. This gives, $ab+\mu a+ \lambda b=0$ and $\lambda
\mu=0$. As $\lambda \neq 0$, we have $\mu=0$. This yields $ab+
\lambda b=0$. Let $f=C(b)$ in $R$. Hence (1) $fb=b$; (2) $yRb=0$
if and only if $yf=0$. Since $(a, \lambda) R_1 (b, \mu)=(0,0)$
for any $r\in R$, we have $(a, \lambda) (r, 0) (b, 0)=(0,0)$. That is
$(arb+\lambda r b, 0)=(0,0)$. Therefore $arb+ \lambda rb=0$  for
any $r\in R$. Consider $0=arb+\lambda rb=arb+\lambda r f b=
(a+\lambda f) rb$ for any $r \in R$. That is $(a+\lambda f) Rb
=0$, and consequently $(a+\lambda f) f=0$.
 Therefore $af = - \lambda f$. Since $g$ is the greatest projection
with this property, we have $f \leq g$, i.e., $f=fg$. Consider
$ey=(-g,1)(b , \mu)=(-g,1)(b , 0)=(-gb+b,
0)=(-gfb+fb,0)=(-fb+fb,0)=(0,0) $. Thus $R_1$ is a p.q.-Baer
$*$-ring.
\end{proof}
A partial solution of {\bf Problem 2} analogous to a partial solution
of {\bf Problem 1} given by Berberian \cite{Ber}, can be obtained
as a corollary.
\begin{corollary} \label{c302'} Let $R$ be a weakly p.q.-Baer $*$-ring. If there
exists an involutary integral domain $K$ such that $R$ is a
$*$-algebra over $K$ and it is torsion free $K$-module, then $R$
can be embedded in a p.q.-Baer $*$-ring with preservation of
central covers.
\end{corollary}
\begin{proof} By \cite[Corollary 4]{Wap}, if $R$ is torsion-free
then $R$ satisfies the condition $(iii)$ in Theorem \ref{th304}.
\end{proof}

\noindent {\bf Remark 2:} Let $A=C_{\infty}(T)\otimes M_2(\mathbb Z_3)$ (external direct product of the
$*$-rings), where $C_{\infty}(T)$ denotes the algebra of all
continuous, complex valued functions on $T$ that vanish at
$\infty$. In \cite{Ber}, it is proved that, $C_{\infty}(T)$ is a
commutative weakly Rickart $C^*$-algebra, 
when $T$ is a non-compact, locally compact, Hausdorff space.
Hence $C_{\infty}(T)$ is a weakly p.q.-Baer $*$-ring.
Since $M_2(\mathbb Z_3)$ is a Baer $*$-ring and hence it is a weakly
p.q.-Baer $*$-ring. 
Therefore $A$ is a weakly p.q.-Baer $*$-ring.
Also $A$ satisfies condition $(iii)$ in the
Theorem \ref{th304}.
 Since there does not exist an integral domain $K$ such that $A$
is torsion free $K$-module, so $A$ does not satisfies the
hypotheses of Corollary \ref{c302'}. However, $A$ satisfies the
hypotheses of the Theorem \ref{th304}. Moreover, $A_1=A \oplus
\mathbb Z$ (operations as in unitification) is a p.q.-Baer
$*$-ring containing $A$ that preserves central covers.\\

\noindent {\bf Acknowledgment:} The first author gratefully
acknowledges the University Grant Commission, New Delhi, India
for the award of Teachers Fellowship under the faculty development
program, during the $XII^{th}$ plan period (2012-1017).

$$\diamondsuit\diamondsuit\diamondsuit$$

\end{document}